\documentclass[11pt,letterpaper,reqno]{amsart}
\usepackage{fullpage}
\usepackage[top=2cm, bottom=4.5cm, left=2.5cm, right=2.5cm]{geometry}
\usepackage{amsmath,amsthm,amsfonts,amssymb,amscd}
\usepackage{geometry,tikz,tikz-cd}
\usepackage{empheq}
\usepackage{lipsum}
\usepackage{lastpage}
\usepackage{thmtools}
\usepackage{enumerate}
\usepackage[shortlabels]{enumitem}
\usepackage{fancyhdr}
\usepackage{mathrsfs}
\usepackage{xcolor}
\usepackage{graphicx}
\usepackage{listings}
\usepackage{bbm}
\usepackage{array}
\usepackage{hyperref}
\usepackage[makeroom]{cancel}

\usepackage[utf8]{inputenc}
\usepackage{todonotes}
\usepackage{comment}
\usepackage{nicematrix}

\usepackage[all]{xy}
\xyoption{matrix}
\xyoption{arrow}

\usetikzlibrary{arrows,decorations.pathmorphing,backgrounds,positioning,fit,petri}

\tikzset{help lines/.style={step=#1cm,very thin, color=gray},
help lines/.default=.5} % draws a grid spaced #1 cm
\tikzset{thick grid/.style={step=#1cm,thick, color=gray},
thick grid/.default=1} % draws a grid spaced #1 cm

\mathtoolsset{showonlyrefs=true}
\numberwithin{figure}{section}
\numberwithin{table}{section}

\theoremstyle{definition}

\theoremstyle{plain}
\newcommand{\thistheoremname}{}
\newtheorem*{genericthm*}{\thistheoremname}
\newenvironment{namedthm*}[1]
  {\renewcommand{\thistheoremname}{#1}%
   \begin{genericthm*}}
  {\end{genericthm*}}

\hypersetup{%
  colorlinks=true,
  linkcolor=blue,
  citecolor=blue,
  linkbordercolor={0 0 1}
}

\makeatletter
\NewDocumentCommand{\sump}{e{_}}
 {%
  \DOTSB
  \mathop{\IfNoValueTF{#1}{\sump@{}}{\sump@{#1}}}%
  \nolimits
 }
\newcommand{\sump@}[1]{\mathpalette\sump@@{#1}}
\newcommand{\sump@@}[2]{%
  \ifx#1\displaystyle
    {\sump@display{#2}}%
  \else
    \sum@\nolimits'_{#2}%
  \fi
}
\newcommand{\sump@display}[1]{%
  \sbox\z@{$\m@th\displaystyle\sum@\nolimits'$}%
  \sbox\tw@{$\m@th\displaystyle\sum@\limits_{#1}$}%
  \sbox\@tempboxa{$\m@th\displaystyle'$}
  \mathop{\sum@\nolimits' \kern-\wd\@tempboxa}\limits_{#1}%
  \ifdim\wd\z@>\wd\tw@
    \kern\dimexpr\wd\z@-\wd\tw@\relax
  \fi
}
\makeatother

\allowdisplaybreaks
 \newtheorem{theorem}{Theorem}[section]
 
 \newtheorem{lemma}[theorem]{Lemma}
 \newtheorem{proposition}[theorem]{Proposition}
 \theoremstyle{definition}
 \newtheorem{definition}{Definition}
 \theoremstyle{definition}
 \newtheorem{example}[theorem]{Example}
 \theoremstyle{remark}
\newtheorem{conjecture}[theorem]{\bf Conjecture}

 \numberwithin{equation}{section}
 \theoremstyle{remark}

\usepackage{latexsym}
\usepackage{amssymb}

\newcommand{\ben}{\begin{equation}}
\newcommand{\een}{\end{equation}}

\allowdisplaybreaks

\DeclareMathOperator{\Sl}{SL}

\DeclareMathOperator{\Span}{Span}
\DeclareMathOperator{\Hom}{Hom}

\DeclareMathOperator{\re}{Re}

\setlist[enumerate]{leftmargin=*,widest=0}
\setlist[itemize]{leftmargin=*,widest=0}
\makeatletter
\def\subsection{\@startsection{subsection}{2}%
  \z@{.5\linespacing\@plus.7\linespacing}{.3\linespacing}%
  {\normalfont\bfseries}}

\def\subsubsection{\@startsection{subsubsection}{3}%
  \z@{.5\linespacing\@plus.7\linespacing}{.3\linespacing}%
  {\normalfont\bfseries}}
\makeatother

\begin{document}

\author{Tianyu Ni}
\address[T. Ni]{School of Mathematical and Statistical Sciences\\
Clemson University\\
Clemson, SC 29634-0975\\
USA}
\email{tianyuni1994math@gmail.com}

\author{Ashley Song}
\address[A. Song]{Department of Mathematics\\
Rice University\\
Houston, TX 77005-1892\\
USA}
\email{as385@rice.edu}
\author{Yanhui Su}
\address[Y. Su]{School of Mathematical and Statistical Sciences\\
Clemson University\\
Clemson, SC 29634-0975\\
USA}
\email{yanhuis@clemson.edu}
\author{Hui Xue}
\address[H. Xue]{School of Mathematical and Statistical Sciences\\
Clemson University\\
Clemson, SC 29634-0975\\
USA}
\email{huixue@clemson.edu}
\author{Amanda Yin}
\address[A. Yin]{Department of
Mathematics\\
The University of Texas at Austin\\
Austin, TX 78712\\
USA} 
\email{amandayin@utexas.edu}

\title{Explicit generators of the space of modular forms}
\begin{abstract} 
Let $S_{\kappa}$ be the space of cusp forms of weight $\kappa$ and level one, and let $S_{\kappa}^{\ast}$ denote its dual space. In this paper, we  find explicit spanning subsets of $S_{\kappa}$ consisting of Rankin-Cohen brackets of Eisenstein series and explicit subsets of periods that span $S_{\kappa}^{\ast}$.
 \end{abstract}

\maketitle
\section{Introduction}
Throughout the paper, we assume that $\kappa\geq12$ is an even integer and adopt the following notation:
\begin{align}
      \kappa&:=12m+2a\quad(m\geq1,\text{ }0\leq a\leq 5),\\
      M_{\kappa}&:={\rm the~space~of~modular~forms~of~weight~}\kappa~{\rm and~level~one,}\\
    S_{\kappa}&:={\rm the ~subspace~of}~M_{\kappa}~{\rm consisting~of~cusp~forms},\\
    S_{\kappa}^{\ast}&:=\Hom_{\mathbb{C}}(S_{\kappa},\mathbb{C})\quad\text{(the dual space of $S_\kappa$)},\\
    \lfloor x\rfloor&:={\rm the~largest~integer~no~greater~than}~x,\\
    \lceil x\rceil&:={\rm the~smallest~integer~no~less~than}~x.
\end{align}
One of the main reasons for the importance of modular forms in number theory is the fact that spaces of modular forms are spanned by forms with rational Fourier coefficients. These coefficients are often arithmetically interesting. For example, divisor functions appear in the Fourier coefficients of Eisenstein series, and the number of representations of integers by quadratic forms is related to the coefficients of theta series. Furthermore,
periods of modular forms, due to their relation to   values of $L$-functions, also play important roles in number theory.

For even $\kappa\geq4$, the Eisenstein series $E_{\kappa}\in M_{\kappa}$ is given by
\begin{equation*}
    E_{\kappa}(z)=1-\frac{2\kappa}{B_{\kappa}}\sum_{n=1}^{\infty}\sigma_{\kappa-1}(n)q^n,
\end{equation*}
where $B_{\kappa}$ is the $\kappa$-th Bernoulli number and $\sigma_{\kappa-1}(n)=\sum_{0<d\mid n}d^{\kappa-1}$.  Applying the result of Eichler-Shimura \cite{Eichlerperiod,Kohnen1984,Manin1973} on periods of modular forms, one can obtain 
\begin{align}M_{\kappa}=\mathbb{C}E_{\kappa}+\mathop{\Span}_{\substack{4\leq\ell\leq \kappa-4}}E_{\ell}\cdot E_{\kappa-\ell}\label{eq:classicalresult}.\end{align}
This motivated a sequence of work on generating the space of modular forms by products of Eisenstein series. Imamo\=glu and Kohnen \cite{Kohnen2005} and Kohnen and Martin \cite{KohnenMartin2008} extended \eqref{eq:classicalresult} to modular forms of prime levels.
Dickson and Neururer \cite{productresult2018} generalized \eqref{eq:classicalresult} to modular forms of levels of the form $N=p^aq^bN^{\prime}$, where $p$ and $q$ are distinct prime numbers and $N^{\prime}$ is a square-free integer such that $\gcd(N^{\prime},pq)=1$.

It is therefore natural to ask if the space of cusp forms can be generated by products of at most two Eisenstein series. Since a product of Eisenstein series is not a cusp form in general, we consider the Rankin-Cohen bracket of Eisenstein series, which is the simplest way to produce a cusp form from modular forms of lower weights.
Let $e \geq 0$ be an integer. Recall that for modular forms $f\in M_{a}(\Gamma)$ and $g\in M_b(\Gamma)$ of weights $a$ and $b$, respectively, where $\Gamma$ is some congruence subgroup, the $e$-th Rankin-Cohen bracket is defined as
\begin{align*}
    [f(z),g(z)]_e := \sum_{r=0}^e (-1)^r\binom{e+a-1}{e-r}\binom{e+b-1}{r}f(z)^{(r)}g(z)^{(e-r)} \in M_{a+b+2e}(\Gamma),
\end{align*}
where $f(z)^{(r)}:=\frac{1}{(2\pi i)^r}\frac{d^r f(z)}{dz^r}$ is the $r$-th normalized derivative of $f(z)$. For $e \geq 1$, we have that $[f,g]_e\in S_{a+b+2e}(\Gamma)$. Note that $[f,g]_0 = fg$; see Cohen \cite[Theorem 7.1]{Cohen'smodularformC_k}.

For example, it is shown in \cite{NiXuevectorRC} that the Rankin-Cohen brackets of vector-valued Eisenstein series span the space of cusp forms as a Hecke module. Xue \cite[Theorem 1.1]{XUeRC2024} showed that the set of  Rankin-Cohen brackets $\{[E_{k},E_{\ell}]_n\}_{k+\ell+2n=\kappa}$ of Eisenstein series spans $S_{\kappa}$ in several cases, namely (1) $n=1, \kappa\geq10$, (2) $n=3,\kappa\geq14$, (3) $n=0, \kappa\geq8$, and (4) $n=2, \kappa\geq12$. It was conjectured that the aforementioned result holds for any $n$, provided that $\kappa\gg n$; see \cite[Conjecture 3.9]{XUeRC2024}. In this paper, we give an affirmative answer to this conjecture. We now state our first result.
\begin{theorem}\label{thm:mainthm}
    Let $n\geq1$ be an integer. If $\kappa\geq 4n + 18$, then the set $\{[E_{k},E_{\ell}]_n\}_{k,\ell}$ for even $k,\ell\geq4$ and $k+\ell+2n=\kappa$ spans the vector space $S_{\kappa}$.
\end{theorem}
Along the way, we also find explicit spanning subsets of the dual space $S_{\kappa}^{\ast}$ consisting of periods of modular forms. For each $0\leq t\leq \kappa-2$, the $t$-th period of $f\in S_{\kappa}$ is given by Manin \cite{Manin1973} as
\begin{align}
    r_t(f):=\int_0^{i\infty}f(z)z^tdz=\frac{t!}{(-2\pi i)^{t+1}}L(f,t+1).\label{eq:defofperiods}
\end{align}
Here, the $L$-function of $f(z)=\sum_{n\geq1}a_f(n)q^n$ is $L(f,s)=\sum_{n\geq1}a_f(n)n^{-s}$, which converges for $\re(s)$ sufficiently large and can be extended analytically to the whole complex plane. Note that each $r_t$ can be regarded as an element of $S_{\kappa}^{\ast}$. The result of Eichler-Shimura \cite{Eichlerperiod,Kohnen1984,Manin1973} asserts that the even periods $r_0,r_2,\dots,r_{\kappa-2}$ span the vector space $S_\kappa^{\ast}$. Similarly, the odd periods $r_1,r_3,\dots,r_{\kappa-3}$ span $S_\kappa^{\ast}$ as well. However, these periods are not linearly independent because both sets of even and odd periods are of size $\approx 6 \dim S_{\kappa}$. In fact, they are subject to many linear dependence relations called the Eichler-Shimura relations \cite{Manin1973}; see \eqref{eq:ES1}, \eqref{eq:ES2}, and \eqref{eq:ES3} below. It is therefore natural to ask which smaller subsets of periods span $S_{\kappa}^{\ast}$. Fukuhara \cite[Theorem 2.2]{Fukuhara07} found the following subset of odd periods that forms a basis for $S_{\kappa}^{\ast}$: the set $\{r_{4i+1}\}_i$ when $4\nmid \kappa$, or the set $\{r_{4i-1}\}_i$ when $4\mid \kappa$, for $1\le i\le \dim S_{\kappa}$. Our second result gives explicit spanning subsets of size $\approx\lceil\frac{3}{2}\dim S_{\kappa}\rceil$, which are different from the aforementioned ones in \cite{Fukuhara07}. %Notice that Fukuhara's bases are taken from the end of odd periods, while our spanning sets are taken from the middle.
\begin{theorem}\label{thm:periodsspanningset}
    Let $\kappa = 12m+2a$ as before, and let $I = \left\{i \in \mathbb{Z} : \lceil \frac{3m}{2} \rceil+1 \leq i \leq 3m + \lfloor\frac{a}{2}\rfloor- 1\right\}$. Then the even periods $\{r_{2i}\}_{i\in I}$ span $S_\kappa^\ast$. Similarly, the odd periods $\{r_{2i-1}\}_{i\in I}$ span $S_\kappa^\ast$ as well.
\end{theorem}  
The paper is organized as follows. In Section \ref{sect:proofofmainthm}, we detail the preliminaries required for the rest of the paper. Assuming Proposition \ref{prop:keyresultofperiods}, we prove Theorem \ref{thm:mainthm}. Section \ref{sect:proofrowreduction} is dedicated to proving Proposition \ref{prop:keyresultofperiods} and Theorem \ref{thm:periodsspanningset}. The idea is to show that a submatrix formed by the Eichler-Shimura relations is non-singular. This is achieved by row reducing it to a lower anti-triangular matrix, unlike the typical use of row reduction for upper triangular matrices. We conclude the paper in Section \ref{sect:discuss} by describing several possible questions for future work.

\section{Proof of Theorem \ref{thm:mainthm}}\label{sect:proofofmainthm}

This section proves Theorem \ref{thm:mainthm}, assuming the necessary results to be proved later.
Let us first recall some basic facts of the Eichler-Shimura theory of periods of modular forms. 
\begin{proposition}[\cite{Eichlerperiod,Kohnen1984,Manin1973}]\label{prop:ESisomorphism} Let $f\in S_{\kappa}$ for $\kappa\geq12$. Then the following hold:
\begin{enumerate}
    \item[\textnormal{(1)}]  If $r_2(f)=r_4(f)=\cdots=r_{\kappa-4}(f)=0$, then $f=0$.
    \item[\textnormal{(2)}] If $r_1(f)=r_3(f)=\cdots=r_{\kappa-3}(f)=0$, then $f=0$.
\end{enumerate} 
\end{proposition}
The Eichler-Shimura relations \cite{Manin1973} are given as follows: for $0\leq t\leq \kappa-2$, we have that
\begin{gather}
    \quad r_t+(-1)^{t}r_{\kappa-2-t}=0,\label{eq:ES1}\tag{$1_t$}
    \\\quad(-1)^{t}r_t+\sum_{\substack{0\leq \ell\leq t\\\ell\equiv0\pmod2}}\binom{t}{\ell}r_{\kappa-2-t+\ell}+\sum_{\substack{0\leq \ell\leq \kappa-2-t\\\ell\equiv t\pmod2}}\binom{\kappa-2-t}{\ell}r_{\ell}=0,\label{eq:ES2}\tag{$2_t$}
    \\\quad\sum_{\substack{1\leq \ell\leq t\\\ell\equiv1\pmod2}}\binom{t}{\ell}r_{\kappa-2-t+\ell}+\sum_{\substack{0\leq \ell\leq \kappa-2-t\\\ell\not\equiv t\pmod2}}\binom{\kappa-2-t}{\ell}r_{\ell}=0.\label{eq:ES3}\tag{$3_t$}
\end{gather}
Using \eqref{eq:ES1}, for $t \leq \frac{\kappa - 2}{2}$, the equation \eqref{eq:ES2} can be rewritten as
\begin{equation*}
\sum_{\substack{0\leq \ell\leq t\\2 \mid \ell}} \left[\binom{\kappa - 2 - t}{\ell} - \binom{t}{t - \ell} \right] r_{\ell}+\sum_{\substack{t<\ell\leq \frac{\kappa - 2}{2}\\2 \mid (\ell - t)}}\left[\binom{\kappa - 2 - t}{\ell} -\binom{\kappa - 2 - t}{\kappa - 2 - \ell} \right]r_{\ell}= 0. \label{eq:ES2'}\tag{$2_t'$}
\end{equation*}
Similarly, using \eqref{eq:ES1}, for $t \leq \frac{\kappa -2}{2}$, the equation \eqref{eq:ES3} can be rewritten as
\begin{equation*}
    \sum_{\substack{1\leq \ell\leq t\\ 2 \nmid \ell}} \left[\binom{\kappa - 2 - t}{\ell} + \binom{t}{t - \ell} \right]r_{\ell}+\sum_{\substack{t < \ell \leq \frac{\kappa - 2}{2}\\ 2\nmid(\ell-t)}}\left[\binom{\kappa - 2 - t}{\ell} +\binom{\kappa - 2 - t}{\kappa - 2 - \ell} \right]r_{\ell}= 0.\label{eq:ES3'}\tag{$3_t'$}
\end{equation*}
The following vanishing criteria for cusp forms will be proved in Section \ref{sect:proofrowreduction}.
\begin{proposition}\label{prop:keyresultofperiods}
    Let $\kappa = 12m +2a$ as before, and let $f\in S_{\kappa}$. Then the following hold: 
    \begin{enumerate}
        \item[\textnormal{(1)}] If $r_{2(\lceil\frac{3m}{2}\rceil+1)}(f)=r_{2(\lceil\frac{3m}{2}\rceil+2)}(f)=\cdots=r_{\kappa-2-2(\lceil\frac{3m}{2}\rceil+1)}(f)=0$, then $f=0$.
        \item[\textnormal{(2)}] If $r_{2\lceil\frac{3m}{2} \rceil+1}(f)=r_{2\lceil\frac{3m}{2} \rceil+3}(f)=\cdots=r_{\kappa-2-(2\lceil\frac{3m}{2} \rceil+1)}(f)=0$, then $f=0$.
    \end{enumerate}
\end{proposition}

We now prove Theorem \ref{thm:mainthm}.
\begin{proof}[Proof of Theorem \ref{thm:mainthm}]
    It suffices to show that if $g\in S_{\kappa}$ is orthogonal to the subspace of $S_{\kappa}$ spanned by $\{[E_{k},E_{\ell}]_n~:~k,\ell\geq4~{\rm even},~k+\ell+2n=\kappa\}$, then $g=0$. Suppose $g=\sum_{j}a_jg_j$, where the $g_j$'s are normalized Hecke eigenforms in $S_{\kappa}$. By the Rankin-Selberg method \cite[Proposition 6]{Zagier1976}, we have that
    \begin{align}
        \langle g_j,[E_{k},E_{\ell}]_n\rangle= C_{k,\ell,n}\cdot r_{k+\ell+n-2}(g_j)r_{\ell+n-1}(g_j),       
    \end{align}
    where $r_{k+\ell+n-2}(g_j)\neq0$ and $C_{k,\ell,n}$ is some non-zero constant depending only on $k,\ell$, and $n$; see \cite[pp. 5--6]{XUeRC2024}. Thus, the orthogonality condition $\langle g,[E_{k},E_{\ell}]_n\rangle=0$ is equivalent to
    \begin{align}
        \sum_{j} a_jr_{k+\ell+n-2}(g_j)r_{\ell+n-1}(g_j)=0. \label{eq:orthogonality}
    \end{align}
    Following the idea of \cite[Theorem 1]{Kohnen2005}, we define $G_g:=\sum_{j}a_jr_{k+\ell+n-2}(g_j)g_j.$
    Hence, \eqref{eq:orthogonality} implies that 
    \begin{align}
        r_{\ell+n-1}(G_g)=\sum_{j}a_jr_{k+\ell+n-2}(g_j)r_{\ell+n-1}(g_j)=0.
    \end{align}
    Here, $\ell+n-1$ runs over $3+n, 5+n,\dots,\kappa-n-5$ because $k, \ell \geq 4$ and $k + \ell + 2n = \kappa$. Since $\kappa=12m+2a\geq 4n+18$ and $2a\leq10$, we have that $3m\geq n+2$. Hence, if $n$ is odd, we get that
$$3+n\leq3m+1\leq2\left(\left\lceil\frac{3m}{2}\right\rceil+1\right) \quad\text{and}\quad\kappa-n-5\geq\kappa-2-2\left(\left\lceil\frac{3m}{2}\right\rceil+1\right).$$    
That is, $G_g$ satisfies the assumption of Proposition \ref{prop:keyresultofperiods} (1), and thus, $G_g=0$. 
Now, if $n$ is even, we have that
$$3+n\leq3m+1\leq2\left\lceil\frac{3m}{2}\right\rceil+1 \quad\text{and}\quad \kappa-n-5\geq\kappa-2-\left(2\left\lceil\frac{3m}{2}\right\rceil+1\right).$$
That is, $G_g$ satisfies the assumption of Proposition \ref{prop:keyresultofperiods} (2), and thus, we also have that $G_g=0$.

Since $r_{k+\ell+n-2}(g_j)\neq0,$ we must have that $a_j=0$ for all $j$. So $g=0$, completing the proof. 
\end{proof}

\section{Proofs of Proposition \ref{prop:keyresultofperiods} and Theorem \ref{thm:periodsspanningset}}\label{sect:proofrowreduction}

Before proving Proposition \ref{prop:keyresultofperiods}, we introduce Stirling numbers of the first and second kinds.
\begin{definition}
    The unsigned Stirling number of the first kind ${\genfrac{[}{]}{0pt}{1}{n}{k}}$ is given by
$$x^{\underline{n}}:=x(x-1)\cdots(x-n+1)=\sum_{k=0}^n(-1)^{n-k}{\genfrac{[}{]}{0pt}{0}{n}{k}}x^k.$$
\end{definition}
The following cases are particularly useful to us \cite[p. 257]{Stirlingnumberbook}:
\begin{align}
{\genfrac{[}{]}{0pt}{0}{n}{k}}=        \begin{cases}
           1, &\quad k=n, \\
           \binom{n}{2}, &\quad k=n-1, \\
           \frac{3n-1}{4}\binom{n}{3},  &\quad k=n-2. \\
        \end{cases}\tag{$\ast$}\label{stirlingnumberofthefirstkind}
\end{align}
\begin{lemma}[{\cite[pp. 81--82]{Stanley2011}\label{lem:binom-polynomial-vanish}}] 
    Let $n \ge 1$. Then for any integer $b\geq0$, we have that
    \begin{align}
        \sum_{t=0}^n (-1)^t \binom{n}{t} t^b = 
        (-1)^n n!{\genfrac{\{}{\}}{0pt}{0}{b}{n}},
    \end{align}
where ${\genfrac{\{}{\}}{0pt}{1}{b}{n}}$ is the Stirling number of the second kind. In particular, we have that
$${\genfrac{\{}{\}}{0pt}{0}{b}{n}}=        \begin{cases}
           0, &\quad 0 \le b \le n-1, \\
           1, &\quad b=n, \\
           \binom{n+1}{2},  &\quad b=n+1. \\
        \end{cases}$$
\end{lemma}
In this section, we use the convention that $\binom{n}{k} = 0$ if $k < 0$ or $k > n$. Additionally, we write $f(x) = \mathcal{O}(g(x))$ if there exist $M, b \in \mathbb{R}$ such that $|f(x)| \leq M|g(x)|$ for $x \geq b$.

\subsection{Proof of Proposition \ref{prop:keyresultofperiods} (1)}\label{subsec:proofof2.1(1)}

Using \eqref{eq:ES1}, we only need to consider \eqref{eq:ES2} for even $t$ with $2 \leq t \leq 6m + 2\lfloor \frac{a}{2}\rfloor - 2$. For each such $t$, applying \eqref{eq:ES1} again, we only need to consider $r_\ell$ with $2 \leq \ell \leq 6m + 2\lfloor \frac{a}{2}\rfloor - 2$ in \eqref{eq:ES2}. We write these $3m + \lfloor \frac{a}{2} \rfloor - 1$ relations as
\begin{equation}
    \begin{psmallmatrix}
        {\color{white}.} & * & {\color{white}.} &{\color{white}.}& * & {\color{white}.}\\
        {\color{white}.}&{\color{white}.}&{\color{white}.}&{\color{white}.}\\
        {\color{white}.} & \mathbf{A} &{\color{white}.} &{\color{white}.}& \mathbf{B} & {\color{white}.}
        \end{psmallmatrix}
    \begin{psmallmatrix}
        r_2 \\
        \vdots \\
        r_{2\lceil \frac{3m}{2} \rceil} \\
        r_{2\lceil \frac{3m}{2} \rceil+2} \\
        \vdots \\
        r_{6m + 2\lfloor\frac{a}{2}\rfloor- 2}
    \end{psmallmatrix} = \mathbf{0}, \label{eq:rel}\tag{$\ast\ast$}
\end{equation}
where $\left(\mathbf{A}\, | \,\mathbf{B}\right)$ gives the last $\lceil\frac{3m}{2}\rceil$ relations in \eqref{eq:ES2}. In particular, $\mathbf{A} = (a_{ij})_{i,j}$ corresponds to the coefficients of $r_2,\dots,r_{2\lceil\frac{3m}{2}\rceil}$. Note that $a_{ij}$ is exactly the coefficient of $r_{\ell}$ in \eqref{eq:ES2} with $\ell = 2j$ and $t = 6m - 2 \lceil\frac{3m}{2}\rceil + 2i + 2\lfloor\frac{a}{2}\rfloor-2$, where $1 \leq i, j \leq \lceil \frac{3m}{2}\rceil$. Applying \eqref{eq:ES2'}, the entries of $\mathbf{A}$ are explicitly given by the following:
\begin{enumerate} 
\item If $m$ is even, then $\ell \leq t$ for all $i, j$, implying that
\begin{equation*}
    a_{ij} = \binom{6m + 2 \lceil\frac{3m}{2}\rceil-2i +2a - 2\lfloor\frac{a}{2}\rfloor}{ 2j} - \binom{ 6m - 2 \lceil\frac{3m}{2}\rceil + 2i + 2\lfloor\frac{a}{2}\rfloor-2}{ 2j}.
\end{equation*}

\item If $m$ is odd, then $\ell \leq t$ for all $i, j$, except for $(i, j) = \left(1, \frac{3m + 1}{2}\right)$ and $a \in \{0, 1\}$. So the entries $a_{ij}$ are given by the formula above, except for $(i, j) = \left(1, \frac{3m + 1}{2}\right)$ and $a \in \{0, 1\}$, in which case
\begin{equation*}
    a_{ij} = \binom{9m +2a-1 }{3m+1} - \binom{9m+2a-1}{2}.
\end{equation*}
\end{enumerate}
By Proposition \ref{prop:ESisomorphism}, we know that if $\mathbf{A}$ is non-singular, then Proposition \ref{prop:keyresultofperiods} (1) holds.

To show that $\mathbf{A}$ is non-singular, we find a non-singular elimination matrix $\mathbf{P}$ such that $\mathbf{P}\mathbf{A}$ is a lower anti-triangular matrix. That is, the entries strictly above the anti-diagonal of $\mathbf{P}\mathbf{A}$ equal $0$. The reason we consider a lower anti-triangular matrix is that the elimination matrices for other triangular matrices (e.g., an upper triangular matrix) do not exhibit obvious patterns. The elimination matrix $\mathbf{P}=({p}_{ij})_{i, j}$ is given by the following:
\begin{enumerate}
\item If $a$ is even, then 
\begin{equation*}
    p_{ij} = \left\{\begin{array}{ll}
        (-1)^{j - i} \binom{2\left\lceil \frac{3m}{2} \right\rceil - 2i + 1}{j - i}, &\quad i \leq j, \\
        0, &\quad i > j.
    \end{array}\right.
\end{equation*}
\item If $a$ is odd, then 
\[ p_{ij} = \begin{cases} 
      (-1)^{j-i}\left[\binom{ 2\lceil\frac{3m}{2}\rceil-2i + 1}{ j-i} - \binom{2\lceil\frac{3m}{2}\rceil -2i+ 1}{j - i - 1}\right],&\quad i \leq j,\\
      0, &\quad i>j. \\
   \end{cases}
\]
\end{enumerate}
\begin{example} If $m = 5$ and $a = 1$, then 
\begin{gather*}
    \mathbf{A}=\begin{psmallmatrix}
        944 & 162184 & 9363816 & 260929812 & 4076349420 & 38910617564 & 239877544004 & 991493847519 \\
        826 & 133931 & 7051044 & 177219757 & 2481248770 & 21090680793 & 114955808408 & 416714805913 \\
        708 & 108870 & 5227222 & 117986427 & 1471399215 & 11058098324 & 52860226020 & 166509721449 \\
        590 & 86545 & 3799620 & 76778715 & 847475772 & 5586727510 & 23206891080 & 62852096805 \\
        472 & 66500 & 2686068 & 48583722 & 472087110 & 2706828502 & 9669234330 & 22239899817 \\
        354 & 48279 & 1813196 & 29524869 & 252225600 & 1248973544 & 3794335944 & 7307136639 \\
        236 & 31426 & 1114674 & 16593929 & 125816405 & 538696340 & 1382317940 & 2198649695 \\
        118 & 15485 & 529452 & 7410195 & 51389130 & 195371085 & 431319000 & 570658635
    \end{psmallmatrix}, \\
    \mathbf{P} = \begin{psmallmatrix}
1 & -14 & 90 & -350 & 910 & -1638 & 2002 & -1430 \\
0 & 1 & -12 & 65 & -208 & 429 & -572 & 429 \\
0 & 0 & 1 & -10 & 44 & -110 & 165 & -132 \\
0 & 0 & 0 & 1 & -8 & 27 & -48 & 42 \\
0 & 0 & 0 & 0 & 1 & -6 & 14 & -14 \\
0 & 0 & 0 & 0 & 0 & 1 & -4 & 5 \\
0 & 0 & 0 & 0 & 0 & 0 & 1 & -2 \\
0 & 0 & 0 & 0 & 0 & 0 & 0 & 1
\end{psmallmatrix}, \\
\mathbf{P}\mathbf{A} = \begin{psmallmatrix}
0 & 0 & 0 & 0 & 0 & 0 & 0 & 1473525 \\
0 & 0 & 0 & 0 & 0 & 0 & 385024 & 31825920 \\
0 & 0 & 0 & 0 & 0 & 100352 & 9024000 & 209574912 \\
0 & 0 & 0 & 0 & 26112 & 2546432 & 64223808 & 658064832 \\
0 & 0 & 0 & 6784 & 715360 & 19540808 & 217203826 & 1188954823 \\
0 & 0 & 1760 & 200128 & 5905630 & 71043609 & 421659184 & 1365831034 \\
0 & 456 & 55770 & 1773539 & 23038145 & 147954170 & 519679940 & 1057332425 \\
118 & 15485 & 529452 & 7410195 & 51389130 & 195371085 & 431319000 & 570658635
\end{psmallmatrix}.
\end{gather*}
See \cite{codeandexamples} for the code for computing $\mathbf{A}$, $\mathbf{P}$, and the anti-diagonal of $\mathbf{PA}$.
\end{example}

Now, we will show that the entries strictly above the anti-diagonal of $\mathbf{P}\mathbf{A}$ equal $0$.
\begin{lemma}\label{lem:even-period-zero}
For $j = 1, 2, \dots, \lceil \frac{3m}{2} \rceil - 1$, if $i < \lceil \frac{3m}{2} \rceil + 1 - j$, then
    \begin{equation*}
        (\mathbf{PA})_{ij}=\sum_{x = 1}^{\lceil \frac{3m}{2} \rceil} p_{ix} a_{xj} 
        = \sum_{x = i}^{\lceil \frac{3m}{2} \rceil} p_{ix} a_{xj} 
        = 0.
    \end{equation*}
\end{lemma}
\begin{proof}
\textbf{Case 1.} Suppose $a$ is even. 
Let $n = 2\lceil\frac{3m}{2}\rceil -2i+1$ and $x'=n+2i-x$. We have that
\begin{align}
\sum_{x = i}^{\lceil \frac{3m}{2} \rceil} p_{ix} a_{xj} 
&=\sum_{x = i}^{\lceil \frac{3m}{2} \rceil} (-1)^{x-i}\left(\binom{2\left\lceil \frac{3m}{2}\right\rceil - 2i + 1}{x-i}\right.\\
&\quad\left.\times\left[\binom{6m + 2 \lceil\frac{3m}{2}\rceil-2x +a}{ 2j} - \binom{ 6m - 2 \lceil\frac{3m}{2}\rceil + 2x + a-2}{ 2j} \right]\right)\\
&=\sum_{x = i}^{\frac{n}{2}+i-\frac{1}{2}} (-1)^{x-i}\binom{n}{x-i}\binom{6m + n +2i-2x +a-1}{ 2j} \\
&\quad-\sum_{x = i}^{\frac{n}{2}+i-\frac{1}{2}} (-1)^{x-i}\binom{n}{x-i} \binom{ 6m - n-2i + 2x + a-1}{ 2j}\\
&=\sum_{x = i}^{\frac{n}{2}+i-\frac{1}{2}} (-1)^{x-i}\binom{n}{x-i}\binom{6m + n +2i-2x +a-1}{ 2j} \\
&\quad+\sum_{x' = \frac{n}{2}+i+\frac{1}{2}}^{n+i} (-1)^{x'-i}\binom{n}{x'-i} \binom{ 6m + n+2i - 2x' + a-1}{ 2j}\\
&=\sum_{x=i}^{n+i} (-1)^{x-i}\binom{n}{x-i} \binom{ 6m + n+2i  - 2x + a-1}{ 2j}\\
&=\sum_{x=0}^{n} (-1)^x\binom{n}{x} \binom{ 6m + n  - 2x + a-1}{2j},\tag{3.3.1}\label{even-period-even-a}
\end{align}
where the last equality holds by the change of variables $x-i \mapsto x$. Note that $\binom{ 6m + n  - 2x + a-1}{ 2j}$ is a polynomial in $x$ of degree $2j$. By assumption, $2j \leq 2 \lceil\frac{3m}{2}\rceil - 2i= n-1$. Hence, by Lemma \ref{lem:binom-polynomial-vanish}, we have that $(\mathbf{P}\mathbf{A})_{ij} = 0$.

\textbf{Case 2.} Suppose $a$ is odd. Let $n = 2\lceil\frac{3m}{2}\rceil -2i+1$ and $x'=n + 2i-x+1$. We have that
\begin{align*}
 \sum_{x = i}^{\lceil \frac{3m}{2} \rceil} p_{ix} a_{xj}&=\sum_{x=i}^{\lceil\frac{3m}{2}\rceil}(-1)^{x-i}\left(\left[\binom{2\lceil\frac{3m}{2}\rceil-2i+1}{x-i}-\binom{2\lceil\frac{3m}{2}\rceil-2i+1}{x-i-1}\right]\right.\\
&\quad\left.\times\left[\binom{6m + 2\lceil\frac{3m}{2}\rceil-2x+a+1}{2j}-\binom{6m-2\lceil\frac{3m}{2}\rceil+2x+a-3}{2j}\right]\right)\\
&=\sum_{x=i}^{\frac{n}{2}+i-\frac{1}{2}}(-1)^{x-i}\left[\binom{n}{x-i}-\binom{n}{x-i-1}\right]
\binom{6m + n+2i-2x+a}{2j}\\
&\quad-\sum_{x=i}^{\frac{n}{2}+i-\frac{1}{2}}(-1)^{x-i}\left[\binom{n}{x-i}-\binom{n}{x-i-1}\right]
\binom{6m-n - 2i+2x+a-2}{2j}\\
&=\sum_{x=i}^{\frac{n}{2}+i-\frac{1}{2}}(-1)^{x-i}\left[\binom{n}{x-i}-\binom{n}{x-i-1}\right]
\binom{6m + n+2i-2x+a}{2j}\\
&\quad+\sum_{x'=\frac{n}{2}+i+\frac{3}{2}}^{n+i+1}(-1)^{x'-i}\left[\binom{n}{x'-i}-\binom{n}{x'-i-1}\right]
\binom{6m+n+2i-2x'+a}{2j}.
\end{align*}
We can rewrite this as the following, noting that the $x=\frac{n}{2}+i+\frac{1}{2}$ term below vanishes:%We can rewrite this as (noting that the $x=\frac{n}{2}+i+\frac{1}{2}$ term vanishes)%note that the x term below vanishes.
\begin{align*}
 \sum_{x = i}^{\lceil \frac{3m}{2} \rceil} {p}_{ix} a_{xj}&=\sum_{x=i}^{n+i+1}(-1)^{x-i}\left[\binom{n}{x-i}-\binom{n}{x-i-1}\right]
\binom{6m+n+2i-2x+a}{2j}\\
&=\sum_{x=0}^{n+1}(-1)^{x}\left[\binom{n}{x}-\binom{n}{x-1}\right]
\binom{6m+n-2x+a}{2j}\tag{3.3.2}\label{even-period-odd-a}\\
&=\sum_{x=0}^{n}(-1)^x\binom{n}{x}
\binom{6m+n-2x+a}{2j} +\sum_{x=0}^{n}(-1)^x\binom{n}{x}
\binom{6m+n-2x+a-2}{2j} \\
&=0,
\end{align*}
where the last equality holds by Lemma \ref{lem:binom-polynomial-vanish} because $2j \leq n - 1$, which finishes the proof.
\end{proof}
Now, in the next lemma, we determine the entries of the anti-diagonal of $\mathbf{P}\mathbf{A}$.
\begin{lemma}\label{lem:even-period-diag}
    Let $y = \lceil \frac{3m}{2} \rceil + 1 - j$. If $m$ is even, then
    \begin{equation*}
        (\mathbf{PA})_{yj} = \sum_{x = 1}^{\lceil \frac{3m}{2} \rceil} p_{yx} a_{xj} = \sum_{x = y}^{\lceil \frac{3m}{2} \rceil} p_{yx} a_{xj} = \left\{\begin{array}{ll}
            2^{2(j - 1)}(12m - 2j + 2a - 1), &\quad a \equiv 0 \pmod 2, \\
            2^{2j - 1}(12m - 2j + 2a - 1), &\quad a \equiv 1 \pmod 2.
    \end{array}\right.
    \end{equation*}
    If $m$ is odd, then $(\mathbf{PA})_{yj}$ is given by the formula above, except for $j = \frac{3m + 1}{2}$ and $a \in \{0, 1\}$, in which case
    \begin{equation*}
        (\mathbf{PA})_{yj}=\sum_{x = 1}^{\lceil \frac{3m}{2} \rceil} p_{yx} a_{xj} = \sum_{x = y}^{\lceil \frac{3m}{2} \rceil} p_{yx} a_{xj} = \left\{\begin{array}{ll}
            2^{3m - 1}(9m - 2) - \binom{9m - 1}{2}, &\quad a = 0,\\
            2^{3m}(9m) - \binom{9m + 1}{2}, &\quad a = 1.
    \end{array}\right.
    \end{equation*}
    In particular, we see that the entries of the anti-diagonal of $\mathbf{P}\mathbf{A}$ are non-zero, which implies that $\mathbf{A}$ is non-singular.
\end{lemma}
\begin{proof} For the ease of exposition, we divide this proof into cases based on the parity of $a$.

    \textbf{Case 1.} Suppose $a$ is even. First, suppose $m$ is even. By \eqref{even-period-even-a}, we have that
    \begin{equation*}
        \sum_{x = y}^{\lceil \frac{3m}{2} \rceil} p_{yx} a_{xj} = \sum_{x = 0}^{2j-1} (-1)^{x}\binom{2j-1}{x}\binom{6m -2x + 2j +a-2}{ 2j},
    \end{equation*}
    where $i = y$. Let $\alpha = 6m + 2j + a-2 $. Note that $\binom{\alpha - 2x}{2j}$ is a polynomial in $x$ of degree $2j$. In particular, we compute the first two coefficients of this polynomial. We have that
    \begin{equation*}
        \binom{\alpha-2x}{2j} = \frac{1}{(2j)!} \left[ 
        A\, x^{2j} + B\, x^{2j-1} + \mathcal{O}(x^{2j-2}) \right],
    \end{equation*}
    where $A = (-2)^{2j} = 2^{2j}$ and
    \begin{equation*}
        B = (-2)^{2j-1}\left[\alpha + (\alpha-1) + \ldots + (\alpha-2j+1)\right] = -2^{2j-1} j (2\alpha-2j+1).
    \end{equation*}
    Applying Lemma \ref{lem:binom-polynomial-vanish} to $n = 2j - 1$, we see that
    \begin{align}
        \sum_{x = y}^{\lceil \frac{3m}{2} \rceil} p_{yx} a_{xj} &= \frac{1}{(2j)!} \left[
        (-1)^{2j-1} (2j-1)! \binom{2j}{2} A + (-1)^{2j-1} (2j-1)!B 
        \right] \\
        &= -\frac{1}{2j} \left[ \binom{2j}{2} A + B \right] \\
        &= \frac{2^{2j-1}}{2j} \left[ -2 \binom{2j}{2} + j\,(2\alpha-2j+1) \right] \\
        &=  2^{2(j-1)} (12m-2j+2a-1) \tag{3.4.1}\label{even-period-even-diag}.
    \end{align}
    This proves the case when $a$ and $m$ are even.
    
    Now, suppose $m$ is odd. Then $(\mathbf{P}\mathbf{A})_{yj}$ is again given by the argument above, except for $j = \frac{3m + 1}{2}$ and $a = 0$. Note that
    \begin{equation*}
        \sum_{x = y}^{\lceil \frac{3m}{2} \rceil} {p}_{yx} a_{xj} = p_{yy}a_{yj} + \sum_{x = y + 1}^{\lceil \frac{3m}{2} \rceil} {p}_{yx} a_{xj}.
    \end{equation*}
    Let $j = \frac{3m + 1}{2}$ and $a = 0$. Then $y = 1$. We define
    \begin{align*}
        S := p_{11}\left[\binom{6m + 2 \lceil\frac{3m}{2}\rceil-2y +2a - 2\lfloor\frac{a}{2}\rfloor}{ 2j} - \binom{ 6m - 2 \lceil\frac{3m}{2}\rceil + 2y + 2\lfloor\frac{a}{2}\rfloor-2}{ 2j}\right] = \binom{9m - 1}{3m + 1}.
    \end{align*}
    By \eqref{even-period-even-diag}, we have that
    \begin{equation*}
        S + \sum_{x = y + 1}^{\lceil \frac{3m}{2} \rceil} {p}_{yx} a_{xj} = 2^{2(j-1)} (12m-2j+2a-1) = 2^{3m - 1}(9m - 2),
    \end{equation*}
    which implies that
    \begin{align*}
        \sum_{x = y}^{\lceil \frac{3m}{2} \rceil} {p}_{yx} a_{xj} &= p_{yy}a_{yj} + \left[2^{3m - 1}(9m - 2) - S\right] \\
        &= \left[\binom{9m - 1}{3m + 1} - \binom{9m - 1}{2}\right] + \left[2^{3m - 1}(9m - 2) - \binom{9m - 1}{3m + 1}\right] \\
        &= 2^{3m - 1}(9m - 2) - \binom{9m - 1}{2}.
    \end{align*}
    This proves the case when $a$ is even and $m$ is odd.
    
   \textbf{Case 2.} Suppose $a$ is odd. First, suppose $m$ is even. By \eqref{even-period-odd-a}, we have that
    \begin{equation*}
        \sum_{x = y}^{\lceil \frac{3m}{2} \rceil} p_{yx} a_{xj} = \sum_{x=0}^{2j}(-1)^x\left[\binom{2j - 1}{x}-\binom{2j - 1}{x-1}\right] \binom{6m - 2x + 2j + a - 1}{2j}.
    \end{equation*}
    Let $\alpha = 6m + 2j + a - 1$. Then $\binom{\alpha - 2x}{2j}$ is a polynomial in $x$ of degree $2j$. We have that
    \begin{equation*}
        \binom{\alpha - 2x}{2j} = \frac{1}{(2j)!}\left[Ax^{2j} + Bx^{2j - 1} + \mathcal{O}(x^{2j - 2})\right],
    \end{equation*}
    where $A = 2^{2j}$ and
    \begin{equation*}
        B = -2^{2j - 1}[\alpha + (\alpha - 1) + \cdots + (\alpha - 2j + 1)] = -2^{2j - 1}j(2\alpha - 2j + 1).
    \end{equation*}
    Applying Lemma \ref{lem:binom-polynomial-vanish} to $n = 2j$, we have that
    \begin{align*}
        \sum_{x = y}^{\lceil \frac{3m}{2} \rceil} p_{yx} a_{xj} &= \sum_{x=0}^{2j - 1}(-1)^x\binom{2j - 1}{x}\binom{\alpha - 2x}{2j}+\sum_{x=0}^{2j-1}(-1)^{x}\binom{2j - 1}{x} \binom{\alpha - 2x - 2}{2j} \\
        &= \frac{A + 2^{2j}}{(2j)!}(-1)^{2j - 1}(2j - 1)!\binom{2j}{2} + \frac{B - 2^{2j - 1}j(2\alpha - 2j - 3)}{(2j)!}(-1)^{2j - 1}(2j - 1)! \\
        &= -2^{2j}(2j - 1) + 2^{2j - 1}(2\alpha - 2j - 1) \\
        &= 2^{2j - 1}(12m - 2j + 2a - 1) \tag{3.4.2}\label{even-period-odd-diag}.
    \end{align*}
    This proves the case when $a$ is odd and $m$ is even.
    
    Now, suppose $m$ is odd. Then $(\mathbf{P}\mathbf{A})_{yj}$ is again given by the argument above, except for $j = \frac{3m + 1}{2}$ and $a = 1$. Note that
    \begin{equation*}
        \sum_{x = y}^{\lceil \frac{3m}{2} \rceil} {p}_{yx} a_{xj} = p_{yy}a_{yj} + \sum_{x = y + 1}^{\lceil \frac{3m}{2} \rceil} {p}_{yx} a_{xj}.
    \end{equation*}
    Let $j = \frac{3m + 1}{2}$ and $a = 1$. Then $y = 1$. We define
    \begin{align*}
        S := p_{11}\left[\binom{6m + 2 \lceil\frac{3m}{2}\rceil-2y +2a - 2\lfloor\frac{a}{2}\rfloor}{ 2j} - \binom{ 6m - 2 \lceil\frac{3m}{2}\rceil + 2y + 2\lfloor\frac{a}{2}\rfloor-2}{ 2j}\right] = \binom{9m + 1}{3m + 1}.
    \end{align*}
    By \eqref{even-period-odd-diag}, we have that
    \begin{equation*}
        S + \sum_{x = y + 1}^{\lceil \frac{3m}{2} \rceil} {p}_{yx} a_{xj} = 2^{2j - 1}(12m - 2j + 2a - 1) = 2^{3m}(9m),
    \end{equation*}
    which implies that
    \begin{align*}
        \sum_{x = y}^{\lceil \frac{3m}{2} \rceil} {p}_{yx} a_{xj} &= p_{yy}a_{yj} + \left[2^{3m}(9m) - S\right] \\
        &= \left[\binom{9m +1 }{3m+1} - \binom{9m+1}{2}\right] + \left[2^{3m}(9m) - \binom{9m + 1}{3m + 1}\right] \\
        &= 2^{3m}(9m) - \binom{9m + 1}{2}.
    \end{align*}
    This proves the case when $a$ and $m$ are odd. Thus, the proof is complete.
\end{proof}
Now, Proposition \ref{prop:keyresultofperiods} (1) follows from Lemmas \ref{lem:even-period-zero} and \ref{lem:even-period-diag}.

\subsection{Proof of Proposition \ref{prop:keyresultofperiods} (2)}\label{subsec:proofof2.1(2)}

Again, using \eqref{eq:ES1}, we only need to consider \eqref{eq:ES3} for even $t$ with $2\leq t\leq 6m + 2\lfloor \frac{a}{2}\rfloor-2$. For each such $t$, applying \eqref{eq:ES1} again, we only need to consider $r_\ell$ with $1\leq \ell\leq 6m+2\lfloor \frac{a}{2}\rfloor -3$ in \eqref{eq:ES3}. We write these $3m+\lfloor \frac{a}{2}\rfloor-1$ relations as
\begin{equation*}
   \begin{psmallmatrix}
        {\color{white}.} & * & {\color{white}.} &{\color{white}.}& * & {\color{white}.}\\
        {\color{white}.}&{\color{white}.}&{\color{white}.}&{\color{white}.}\\
        {\color{white}.} & \mathbf{A} &{\color{white}.} &{\color{white}.}& \mathbf{B} & {\color{white}.}
        \end{psmallmatrix}
    \begin{psmallmatrix}
        r_1 \\
        \vdots \\
        r_{2\lceil \frac{3m}{2} \rceil-1} \\
        r_{2\lceil \frac{3m}{2} \rceil+1} \\
        \vdots \\
        r_{6m + 2\lfloor\frac{a}{2}\rfloor- 3}
    \end{psmallmatrix} = \mathbf{0},
\end{equation*}
where $\left(\mathbf{A}\, | \,\mathbf{B}\right)$ gives the last $\lceil\frac{3m}{2} \rceil$ relations in \eqref{eq:ES3}. In particular, $\mathbf{A}=(a_{ij})_{i,j}$ corresponds to the coefficients of $r_1,\dots,r_{2\lceil\frac{3m}{2}\rceil-1}$. Note that $a_{ij}$ is exactly the coefficient of $r_{\ell}$ in \eqref{eq:ES3} with $\ell = 2j - 1$ and $t = 6m -2\lceil\frac{3m}{2}\rceil + 2i + 2\lfloor\frac{a}{2}\rfloor  -2$, where $1\leq i,j\leq \lceil\frac{3m}{2}\rceil$. Applying \eqref{eq:ES3'}, the entries of $\mathbf{A}$ are explicitly given by the following:
\begin{enumerate}
\item If $m$ is even, then $\ell \leq t$ for all $i, j$, implying that
\begin{equation*}
    a_{ij} =\binom{6m + 2 \lceil\frac{3m}{2}\rceil-2i +2a - 2\lfloor\frac{a}{2}\rfloor}{ 2j-1} + \binom{ 6m - 2 \lceil\frac{3m}{2}\rceil + 2i + 2\lfloor\frac{a}{2}\rfloor-2}{ 2j-1}.
\end{equation*}
\item If $m$ is odd, then $\ell \leq t$ for all $i, j$, except for $(i, j) = \left(1, \frac{3m + 1}{2}\right)$ and $a \in \{0, 1\}$. So the entries $a_{ij}$ are given by the formula above, except for $(i, j) = \left(1, \frac{3m + 1}{2}\right)$ and $a \in \{0, 1\}$, in which case
\[ a_{ij} =
      \binom{9m+2a - 1}{ 3m} +\binom{9m+2a- 1}{1}.
\]
\end{enumerate}

The elimination matrix $\mathbf{P}=(p_{ij})_{i,j}$ is given by the following: 
\begin{enumerate}
    \item 
If $a$ is even, then 
\[ p_{ij} = \begin{cases} 
      (-1)^{j-i}\left[\binom{ 2\lceil \frac{3m}{2}\rceil-2i - 1}{ j-i} - \binom{2\lceil\frac{3m}{2}\rceil-2i -1}{ j-i  - 2}\right],&\quad i \leq j,\\
      0, &\quad i>j. \\
   \end{cases}
\]
\item If $a$ is odd, then 
\begin{equation*}
    p_{ij} = \left\{\begin{array}{ll}
        (-1)^{j - i} \frac{2\left(\lceil \frac{3m}{2} \rceil - j + 1\right)^2}{\left(2\lceil \frac{3m}{2} \rceil - i - j + 2\right)\left(\lceil \frac{3m}{2} \rceil - i + 1\right)} \binom{2\left\lceil \frac{3m}{2} \right\rceil - 2i + 1}{j - i}, &\quad i \leq j, \\
        0, &\quad i > j.
    \end{array}\right.
\end{equation*}
\end{enumerate}

\begin{example}
If $m = 5$ and $a = 1$, then
\begin{gather*}
   \mathbf{A}= \begin{psmallmatrix}
    60 & 15544 & 1372756 & 53528112 & 1101718332 & 13340783560 & 101766230804 & 511738760590 \\
    60 & 13804 & 1090376 & 38332008 & 708941948 & 7669343500 & 51915526992 & 229911617072 \\
    60 & 12296 & 859236 & 27010152 & 445940430 & 4280593200 & 25518739848 & 98672428432 \\
    60 & 11020 & 673512 & 18721080 & 273606840 & 2311969400 & 12033300400 & 40225360560 \\
    60 & 9976 & 528276 & 12790800 & 163509060 & 1204027720 & 5415447716 & 15471457104 \\
    60 & 9164 & 419496 & 8693784 & 95450784 & 603301440 & 2313285744 & 5569210064 \\
    60 & 8584 & 344036 & 6037416 & 55575806 & 293823920 & 938384360 & 1863693680 \\
    60 & 8236 & 299656 & 4549896 & 34955700 & 150498660 & 384815760 & 603164880
    \end{psmallmatrix},\\
  \mathbf{P} = \begin{psmallmatrix}
    1 & -\frac{49}{4} & \frac{135}{2} & -\frac{875}{4} & 455 & -\frac{2457}{4} & \frac{1001}{2} & -\frac{715}{4} \\
    0 & 1 & -\frac{72}{7} & \frac{325}{7} & -\frac{832}{7} & \frac{1287}{7} & -\frac{1144}{7} & \frac{429}{7} \\
    0 & 0 & 1 & -\frac{25}{3} & \frac{88}{3} & -55 & 55 & -22 \\
    0 & 0 & 0 & 1 & -\frac{32}{5} & \frac{81}{5} & -\frac{96}{5} & \frac{42}{5} \\
    0 & 0 & 0 & 0 & 1 & -\frac{9}{2} & 7 & -\frac{7}{2} \\
    0 & 0 & 0 & 0 & 0 & 1 & -\frac{8}{3} & \frac{5}{3} \\
    0 & 0 & 0 & 0 & 0 & 0 & 1 & -1 \\
    0 & 0 & 0 & 0 & 0 & 0 & 0 & 1
    \end{psmallmatrix},\\
    \mathbf{PA} = \begin{psmallmatrix}
0 & 0 & 0 & 0 & 0 & 0 & 0 & 1413166 \\
0 & 0 & 0 & 0 & 0 & 0 & \frac{2555904}{7} & \frac{221859840}{7} \\
0 & 0 & 0 & 0 & 0 & \frac{281600}{3} & \frac{26540032}{3} & 217470336 \\
0 & 0 & 0 & 0 & \frac{119808}{5} & 2444800 & \frac{325683712}{5} & \frac{3564522336}{5} \\
0 & 0 & 0 & 6048 & 666224 & 19193370 & 227497228 & 1344790496 \\
0 & 0 & \frac{4480}{3} & 177168 & \frac{16524404}{3} & \frac{211806260}{3} & \frac{1356861152}{3} & \frac{4813905152}{3} \\
0 & 348 & 44380 & 1487520 & 20620106 & 143325260 & 553568600 & 1260528800 \\
60 & 8236 & 299656 & 4549896 & 34955700 & 150498660 & 384815760 & 603164880
\end{psmallmatrix}.
\end{gather*}
See \cite{codeandexamples} for the code for computing $\mathbf{A}$, $\mathbf{P}$, and the anti-diagonal of $\mathbf{PA}$.

\end{example} 
\begin{lemma}\label{lem:odd-period-zero}
    For $j = 1, 2, \dots, \lceil \frac{3m}{2} \rceil - 1$, if $i < \lceil \frac{3m}{2} \rceil + 1 - j$, then
    \begin{equation*}
  (\mathbf{PA})_{ij}   =   \sum_{x = 1}^{\lceil \frac{3m}{2} \rceil} p_{ix} a_{xj} = \sum_{x = i}^{\lceil \frac{3m}{2} \rceil} p_{ix} a_{xj} = 0.
    \end{equation*}
\end{lemma}
\begin{proof}

The proof of Lemma \ref{lem:odd-period-zero} is similar to the proof of Lemma \ref{lem:even-period-zero}.

\textbf{Case 1.} Suppose $a$ is even. Let $n = 2\lceil\frac{3m}{2}\rceil - 2i-1$ and $x'=n +2i-x+2$. We have that
\begin{align}
\sum_{x = i}^{\lceil \frac{3m}{2} \rceil} p_{ix} a_{xj}
&=\sum_{x=i}^{\lceil\frac{3m}{2}\rceil}(-1)^{x-i}\left(\left[\binom{2\lceil\frac{3m}{2}\rceil-2i-1}{x-i}-\binom{2\lceil\frac{3m}{2}\rceil-2i-1}{x-i-2}\right]\right.\\
&\quad\left.\times\left[\binom{6m + 2\lceil\frac{3m}{2}\rceil-2x+a}{2j-1}+\binom{6m-2\lceil\frac{3m}{2}\rceil+2x+a-2}{2j-1}\right]\right)\\
&=\sum_{x=i}^{\frac{n}{2}+i+\frac{1}{2}}(-1)^{x-i}\left[\binom{n}{x-i}-\binom{n}{x-i-2}\right]
\binom{6m + n+2i-2x+a+1}{2j-1}\\
&\quad+\sum_{x=i}^{\frac{n}{2}+i+\frac{1}{2}}(-1)^{x-i}\left[\binom{n}{x-i}-\binom{n}{x-i-2}\right]
\binom{6m-n-2i+2x+a-3}{2j-1}\\
&=\sum_{x=i}^{\frac{n}{2}+i+\frac{1}{2}}(-1)^{x-i}\left[\binom{n}{x-i}-\binom{n}{x-i-2}\right]
\binom{6m + n+2i-2x+a+1}{2j-1}\\
&\quad+\sum_{x'=\frac{n}{2}+i+\frac{3}{2}}^{n+i+2}(-1)^{x'-i}\left[\binom{n}{x'-i}-\binom{n}{x'-i-2}\right]
\binom{6m + n+2i-2x+a+1}{2j-1}\\
&=\sum_{x=i}^{n+i+2}(-1)^{x-i}\left[\binom{n}{x-i}-\binom{n}{x-i-2}\right]
\binom{6m + n+2i-2x+a+1}{2j-1} \\
&= \sum_{x=0}^{n+2}(-1)^{x}\left[\binom{n}{x}-\binom{n}{x-2}\right]
\binom{6m + n-2x+a+1}{2j-1}\tag{3.6.1}\label{odd-period-even-a}\\
&=\sum_{x=0}^{n}(-1)^x\binom{n}{x}
\left[\binom{6m+n-2x+a}{2j-1}-\binom{6m+n-2x+a-3}{2j-1}\right]\\
&=0,
\end{align}
where the last equality holds by Lemma \ref{lem:binom-polynomial-vanish}, since $\binom{6m+n-2x+a}{2j-1}-\binom{6m+n-2x+a-3}{2j-1}$ is a polynomial in $x$ of degree $2j-2\leq n-1$.

\textbf{Case 2.} Suppose $a$ is odd. Let $n =2\lceil \frac{3m}{2} \rceil-2i+2$. We have that
\begin{align*}
    \sum_{x = i}^{\lceil \frac{3m}{2} \rceil} p_{ix} a_{xj}
    &=\sum_{x = i}^{\lceil \frac{3m}{2} \rceil} (-1)^{x - i} \left(\frac{2\left(\lceil \frac{3m}{2} \rceil - x + 1\right)^2}{\left(2\lceil \frac{3m}{2} \rceil - i - x + 2\right)\left(\lceil \frac{3m}{2} \rceil - i + 1\right)}\binom{2\left\lceil \frac{3m}{2} \right\rceil - 2i + 1}{x - i}\right.\\
    &\quad \left. \times \left[\binom{6m + 2 \lceil\frac{3m}{2}\rceil-2x +a+1}{ 2j-1} + \binom{ 6m - 2 \lceil\frac{3m}{2}\rceil + 2x + a-3}{ 2j-1} \right]\right) \\
    &=\sum_{x = i}^{\frac{n}{2}+i-1} (-1)^{x - i} \frac{\left(n+2i-2x\right)^2}{n\left(n+i-x\right)}\binom{n-1}{x - i}\binom{6m + n+2i-2x +a-1}{ 2j-1} \\
    &\quad + \sum_{x = i}^{\frac{n}{2}+i-1} (-1)^{x - i} \frac{\left(n+2i-2x\right)^2}{n\left(n+i-x\right)}\binom{n-1}{x - i}\binom{ 6m -n-2i+ 2x + a-1}{ 2j-1}.
    \end{align*}
    Let $x' = n-x+2i$. We have that $\sum_{x = i}^{\lceil \frac{3m}{2} \rceil} p_{ix} a_{xj} = $
    \begin{align*}
    &\sum_{x = i}^{\frac{n}{2}+i-1} (-1)^{x - i} \frac{\left(n+2i-2x\right)^2}{n\left(n+i-x\right)}\binom{n-1}{x - i}\binom{6m + n+2i-2x +a-1}{ 2j-1} \\
    &\quad + \sum_{x' = \frac{n}{2}+i+1}^{n+i} (-1)^{x' - i} \frac{\left(-n-2i+2x'\right)^2}{n\left(-i+x'\right)}\binom{n-1}{x' - i-1}\binom{ 6m +n+2i- 2x' + a-1}{ 2j-1}  \\
    =&\sum_{x = i}^{\frac{n}{2}+i-1} (-1)^{x - i} \frac{\left(n+2i-2x\right)^2}{n\left(n+i-x\right)}\binom{n-1}{x - i}\binom{6m + n+2i-2x +a-1}{ 2j-1}\\
    &\quad + \sum_{x' = \frac{n}{2}+i+1}^{n+i-1} (-1)^{x' - i} \frac{\left(n+2i-2x'\right)^2}{n\left(x'-i\right)}\cdot \frac{x'-i}{n+i-x}\binom{n-1}{x' - i}\binom{ 6m +n+2i- 2x' + a-1}{ 2j-1}  \\
    &\quad + \binom{6m-n+a-1}{2j-1}\\
    =& \sum_{x = i}^{n+i-1} (-1)^{x - i} \frac{\left(n+2i-2x\right)^2}{n\left(n+i-x\right)}\binom{n-1}{x - i}\binom{6m + n+2i-2x +a-1}{ 2j-1} + \binom{6m-n+a-1}{2j-1},
\end{align*}
where the last equality holds because the $x = \frac{n}{2}+i$ term equals $0$. By the change of variables $x-i \mapsto x$, we have that
\begin{align*}
     \sum_{x = i}^{\lceil \frac{3m}{2} \rceil} p_{ix} a_{xj}& =\sum_{x = 0}^{n-1} (-1)^{x} \frac{\left(n-2x\right)^2}{n\left(n-x\right)}\binom{n-1}{x}\binom{6m + n-2x +a-1}{ 2j-1} + \binom{6m-n+a-1}{2j-1}\\
    &= \sum_{x = 0}^{n-1} (-1)^{x} \frac{\left(n-2x\right)^2}{n\left(n-x\right)}\cdot\frac{n-x}{n}\binom{n}{x}\binom{6m + n-2x +a-1}{ 2j-1} + \binom{6m-n+a-1}{2j-1}\\
    &= \sum_{x = 0}^{n-1} (-1)^{x} \frac{\left(n-2x\right)^2}{n^2}\binom{n}{x}\binom{6m + n-2x +a-1}{ 2j-1} + \binom{6m-n+a-1}{2j-1}\\
    &= \sum_{x = 0}^{n} (-1)^{x} \frac{\left(n-2x\right)^2}{n^2}\binom{n}{x}\binom{6m + n-2x +a-1}{ 2j-1} -\binom{6m-n+a-1}{2j-1}\\
    &\quad + \binom{6m-n+a-1}{2j-1}\\
    &= \sum_{x = 0}^{n} (-1)^{x} \frac{\left(n-2x\right)^2}{n^2}\binom{n}{x}\binom{6m + n-2x +a-1}{ 2j-1} \tag{3.6.2}\label{odd-period-odd-a}\\
    &= 0,
\end{align*}
where the last equality holds by Lemma \ref{lem:binom-polynomial-vanish}, since $\left(n-2x\right)^2\binom{6m + n-2x +a-1}{ 2j-1}$ is a polynomial in $x$ of degree $2j+1\leq n-1$. This completes the proof.
\end{proof}
\begin{lemma}\label{lem:odd-period-diag}
    Let $y = \lceil \frac{3m}{2} \rceil + 1 - j$. If $m$ is even, then
    \begin{equation*}
        (\mathbf{PA})_{yj} = \sum_{x = 1}^{\lceil \frac{3m}{2} \rceil} p_{yx} a_{xj} = \sum_{x = y}^{\lceil \frac{3m}{2} \rceil} p_{yx} a_{xj} = \left\{\begin{array}{ll}
            2^{2(j - 1)}(12m - 2j + 2a), &\quad a \equiv 0 \pmod 2, \\
            \frac{2^{2(j - 1)}(12m - 2j + 2a)(2j - 1)}{j}, &\quad a \equiv 1 \pmod 2.
    \end{array}\right.
    \end{equation*}
    If $m$ is odd, then $(\mathbf{PA})_{yj}$ is given by the formula above, except for $j = \frac{3m+1}{2}$ and $a\in \{0,1\}$, in which case
    \begin{equation*}
        (\mathbf{PA})_{yj} =\sum_{x = 1}^{\lceil \frac{3m}{2} \rceil} p_{yx} a_{xj} = \sum_{x = y}^{\lceil \frac{3m}{2} \rceil} p_{yx} a_{xj} = \left\{\begin{array}{ll}
            2^{3m - 1}(9m - 1) + (9m - 1), &\quad a = 0, \\
            \frac{2^{3m}(9m + 1)(3m)}{3m + 1} + (9m + 1), &\quad a = 1.
    \end{array}\right.
    \end{equation*}
\end{lemma}

\begin{proof} The proof of Lemma \ref{lem:odd-period-diag} is similar to the proof of Lemma \ref{lem:even-period-diag}. For the ease of exposition, we divide this proof into cases based on the parity of $a$.

    \textbf{Case 1.} Suppose $a$ is even. First, suppose $m$ is even. By \eqref{odd-period-even-a}, we have that
    \begin{equation*}
        \sum_{x = y}^{\lceil \frac{3m}{2} \rceil} p_{yx} a_{xj} = \sum_{x = 0}^{2j - 1} (-1)^{x}\left[\binom{2j - 3}{x} - \binom{2j - 3}{x  - 2}\right]\binom{6m - 2x + 2j + a - 2}{ 2j-1}.
    \end{equation*}
    Let $\alpha = 6m + 2j + a - 2$. Then $\binom{\alpha - 2x}{2j - 1}$ is a polynomial in $x$ of degree $2j - 1$. Note that
    \begin{equation*}
        \binom{\alpha - 2x}{2j - 1} = \frac{1}{(2j - 1)!}\left[Ax^{2j - 1} + Bx^{2j - 2} + Cx^{2j - 3} + \mathcal{O}(x^{2j - 4})\right],
    \end{equation*}
    where $A = -2^{2j - 1}$ and
    \begin{align}
        B = (-2)^{2j - 2}[\alpha + (\alpha - 1) + \cdots + (\alpha - 2j + 2)] = 2^{2j - 2}(2j - 1)(\alpha - j + 1). \tag{3.7.1}\label{odd-period-odd-a-coeff}
    \end{align}
    To compute $C$, we consider unsigned Stirling numbers of the first kind. We have that
    \begin{equation*}
        \binom{\alpha - 2x}{2j - 1} = \frac{1}{(2j - 1)!}(\alpha - 2x)^{\underline{2j - 1}} = \frac{1}{(2j - 1)!}\sum_{k = 0}^{2j - 1} (-1)^{2j-1-k}{\genfrac{[}{]}{0pt}{0}{2j-1}{k}} (\alpha - 2x)^k.
    \end{equation*}
By \eqref{stirlingnumberofthefirstkind}, we have that
    \begin{align*}
        C  &= \genfrac{[}{]}{0pt}{}{2j - 1}{2j - 3} (-2)^{2j - 3} - \genfrac{[}{]}{0pt}{}{2j - 1}{2j - 2}\binom{2j - 2}{2j - 3} (-2)^{2j - 3} \alpha + \genfrac{[}{]}{0pt}{}{2j - 1}{2j - 1}\binom{2j - 1}{2j - 3} (-2)^{2j - 3} \alpha^2 \\
        &= \frac{3j - 2}{2} \binom{2j - 1}{3} (-2)^{2j - 3} - \binom{2j - 1}{2}(2j - 2)(-2)^{2j - 3} \alpha + (2j - 1)(j - 1)(-2)^{2j - 3} \alpha^2 \\
        &= -\frac{2^{2j - 4}(2j - 1)(j - 1)}{3}(6j^2 - 13j - 12j\alpha + 12\alpha + 6\alpha^2 + 6).
    \end{align*}
    Applying Lemma \ref{lem:binom-polynomial-vanish} to $n = 2j-1$, we have that
    \begin{align*}
        \sum_{x = y}^{\lceil \frac{3m}{2} \rceil} p_{yx} a_{xj} &= \sum_{x = 0}^{2j - 3} (-1)^{x}\binom{2j - 3}{x}\binom{\alpha - 2x}{2j - 1} - \sum_{x = 0}^{2j - 3} (-1)^{x}\binom{2j - 3}{x}\binom{\alpha - 4 - 2x}{2j - 1} \\
        &= \frac{B - 2^{2j - 2}(2j - 1)(\alpha - j - 3)}{(2j - 1)!} (-1)^{2j - 3}(2j - 3)!\binom{2j - 2}{2} \\
        &\hspace{-5pt}\quad  + \frac{3C + 2^{2j - 4}(2j - 1)(j - 1)(6j^2 + 35j - 12j\alpha - 36\alpha + 6\alpha^2 + 54)}{3(2j - 1)!}(-1)^{2j - 3}(2j - 3)! \\
        &= -2^{2j - 1}(2j - 3) + 2^{2j - 1}(\alpha - j - 1) \\
        &= 2^{2(j - 1)}(12m - 2j + 2a).\tag{3.7.2}\label{odd-period-even-a-diag}
    \end{align*}
    This proves the case when $a$ and $m$ are even.
    
    Now, suppose $m$ is odd. Then $(\mathbf{P}\mathbf{A})_{yj}$ is again given by the argument above, except for $j = \frac{3m + 1}{2}$ and $a=0$. Note that
    \begin{equation*}
        \sum_{x = y}^{\lceil \frac{3m}{2} \rceil} p_{yx} a_{xj} = p_{yy}a_{xj} +  \sum_{x = y+1}^{\lceil \frac{3m}{2} \rceil} p_{yx} a_{xj},
    \end{equation*}
    Let $j = \frac{3m+1}{2}$ and $a=0$. Then $y = 1$. We define
    \begin{equation*}
        S := {p}_{11}\left[\binom{6m + 2 \lceil\frac{3m}{2}\rceil-2 +2a - 2\lfloor\frac{a}{2}\rfloor}{ 2j-1} + \binom{ 6m - 2 \lceil\frac{3m}{2}\rceil + 2 + 2\lfloor\frac{a}{2}\rfloor-2}{ 2j-1}\right]  = \binom{9m - 1}{3m}.
    \end{equation*}
    By \eqref{odd-period-even-a-diag}, we have that
    \begin{equation*}
        S + \sum_{x = y+1}^{\lceil \frac{3m}{2} \rceil} p_{yx} a_{xj} = 2^{2\left(\frac{3m + 1}{2} - 1\right)}\left[12m - 2\left(\frac{3m + 1}{2}\right)\right] = 2^{3m - 1}(9m - 1),
    \end{equation*}
    which implies that
    \begin{align*}
        \sum_{x = y}^{\lceil \frac{3m}{2} \rceil} p_{yx} a_{xj} &= p_{yy}a_{xj} +  \left[2^{3m - 1}(9m - 1) - S\right]\\
        &= \left[\binom{9m-1}{3m} + 9m-1\right] + \left[2^{3m - 1}(9m - 1) - \binom{9m - 1}{3m}\right]\\
        &= 2^{3m - 1}(9m - 1) + (9m - 1).
    \end{align*}
    This proves the case when $a$ is even and $m$ is odd.

 \textbf{Case 2.} Suppose $a$ is odd. First, suppose $m$ is even. By \eqref{odd-period-odd-a}, we have that
    \begin{equation*}
        \sum_{x = y}^{\lceil \frac{3m}{2} \rceil} {p}_{yx} a_{xj} =\sum_{x = 0}^{n} (-1)^{x} \frac{\left(n-2x\right)^2}{n^2}\binom{n}{x}\binom{6m + n-2x +a-1}{ 2j-1},
    \end{equation*} where $n = 2\lceil\frac{3m}{2}\rceil -2y+2 = 2j$. Let $\alpha = 6m +n+a-1$. Then $\frac{(n-2x)^2}{n^2}\binom{\alpha -2x}{2j-1}$ is a polynomial in $x$ of degree $2j+1$. We have that 
    \begin{align*}
        \frac{(n-2x)^2}{n^2}\binom{\alpha -2x}{2j-1} &= \frac{(n-2x)^2}{n^2(2j-1)!}\left[Ax^{2j-1}+Bx^{2j-2}+\mathcal{O}(x^{2j-3})\right]\\
        &= \frac{4Ax^{2j+1}+(4B-4nA)x^{2j} + \mathcal{O}(x^{2j-1})}{n^2(2j-1)!}.
    \end{align*}
    By \eqref{odd-period-odd-a-coeff}, we have that $A = -2^{2j - 1}$ and
    \begin{equation*}
        B = 2^{2j - 2}(2j - 1)(\alpha - j + 1).
    \end{equation*}
    Applying Lemma \ref{lem:binom-polynomial-vanish} to $n = 2j$, we have that
\begin{align*}
    \sum_{x = y}^{\lceil \frac{3m}{2} \rceil} p_{yx} a_{xj} &= \frac{(-1)^{2j}(2j)!\binom{2j+1}{2}4A +(-1)^{2j}(2j)!(4B-4nA)}{n^2(2j-1)!}\\
    &= \frac{2j\left[(2j + 1)jA + (B-2jA)\right]}{j^2}\\
%    &= \frac{-2^{2j}(2j - 1)j + 2^{2j-1}(2j-1)(\alpha-j+1)}{j} \\
    &= \frac{2^{2(j - 1)}(12m - 2j + 2a)(2j - 1)}{j}.\tag{3.7.3}\label{odd-period-odd-a-diag}
\end{align*}
This proves the case when $a$ is odd and $m$ is even.

    Now, suppose $m$ is odd. Then $(\mathbf{P}\mathbf{A})_{yj}$ is again given by the argument above, except for $j = \frac{3m + 1}{2}$ and $a=1$. Note that
    \begin{equation*}
        \sum_{x = y}^{\lceil \frac{3m}{2} \rceil} p_{yx} a_{xj} = p_{yy}a_{xj} +  \sum_{x = y+1}^{\lceil \frac{3m}{2} \rceil} p_{yx} a_{xj},
    \end{equation*}
    Let $j = \frac{3m+1}{2}$ and $a=1$. Then $y = 1$. We define
    \begin{equation*}
        S := {p}_{11} \left[\binom{6m + 2 \lceil\frac{3m}{2}\rceil-2 +2a - 2\lfloor\frac{a}{2}\rfloor}{ 2j-1} + \binom{ 6m - 2 \lceil\frac{3m}{2}\rceil + 2 + 2\lfloor\frac{a}{2}\rfloor-2}{ 2j-1}\right] = \binom{9m + 1}{3m}.
    \end{equation*}
    By \eqref{odd-period-odd-a-diag}, we have that
    \begin{align*}
        S + \sum_{x = y+1}^{\lceil \frac{3m}{2} \rceil} p_{yx} a_{xj} &= \frac{2^{2\left(\frac{3m + 1}{2} - 1\right)}\left[12m - 2\left(\frac{3m + 1}{2}\right) + 2(1)\right]\left[2\left(\frac{3m + 1}{2}\right) - 1\right]}{\frac{3m + 1}{2}}  \\
        &= \frac{2^{3m}(9m + 1)(3m)}{3m + 1},
    \end{align*}
    which implies that
    \begin{align*}
        \sum_{x = y}^{\lceil \frac{3m}{2} \rceil} p_{yx} a_{xj} &= p_{yy}a_{xj} +  \left[\frac{2^{3m}(9m + 1)(3m)}{3m + 1} - S\right]\\
        &= \left[\binom{9m+1}{3m} + (9m+1)\right] + \left[\frac{2^{3m}(9m + 1)(3m)}{3m + 1} - \binom{9m + 1}{3m}\right]\\
        &=\frac{2^{3m}(9m + 1)(3m)}{3m + 1} + (9m + 1).
    \end{align*}
    This proves the case when $a$ and $m$ are odd. Thus, the proof is complete.
\end{proof}
    
Now, Proposition \ref{prop:keyresultofperiods} (2) follows from Lemmas \ref{lem:odd-period-zero}
and \ref{lem:odd-period-diag}.

\subsection{Proof of Theorem \ref{thm:periodsspanningset}}

\begin{proof}Now, we want to show that $\{r_{2i}\}_{i\in I}=\{r_{2\lceil \frac{3m}{2} \rceil+2}, \dots, r_{6m + 2\lfloor\frac{a}{2}\rfloor- 2}\}$ spans $S_{\kappa}^{\ast}$.
Recall the $3m + \lfloor \frac{a}{2} \rfloor - 1$ relations \eqref{eq:rel} for the even periods, which imply that
\begin{equation*}
    \mathbf{A} \begin{psmallmatrix}
        r_2 \\
        \vdots \\
        r_{2\lceil \frac{3m}{2} \rceil}
    \end{psmallmatrix} + \mathbf{B} \begin{psmallmatrix}
        r_{2\lceil \frac{3m}{2} \rceil+2} \\
        \vdots \\
        r_{6m + 2\lfloor\frac{a}{2}\rfloor- 2}
    \end{psmallmatrix} = \mathbf{0}.
\end{equation*}
Since $\mathbf{A}$ is non-singular, we have that
\begin{align}
    \begin{psmallmatrix}
        r_2 \\
        \vdots \\
        r_{2\lceil \frac{3m}{2} \rceil}
    \end{psmallmatrix} = -\mathbf{A}^{-1}\mathbf{B}\begin{psmallmatrix}
        r_{2\lceil \frac{3m}{2} \rceil+2} \\
        \vdots \\
        r_{6m + 2\lfloor\frac{a}{2}\rfloor- 2}
    \end{psmallmatrix}.
\end{align}
That is, $r_2,\dots , r_{2\lceil \frac{3m}{2} \rceil}$ can be expressed as linear combinations of 
$r_{2\lceil \frac{3m}{2} \rceil+2}, \dots, r_{6m + 2\lfloor\frac{a}{2}\rfloor- 2}$.
On the other hand, using \eqref{eq:ES1} and the fact that $r_2, \dots, r_{\kappa - 4}$ span $S_\kappa^\ast$, we see that 
$r_2,\dots,r_{6m + 2\lfloor\frac{a}{2}\rfloor- 2}$ span $S_{\kappa}^{\ast}$. Thus, $\{r_{2i}\}_{i\in I}$ spans $S_\kappa^*$. The proof that $\{r_{2i-1}\}_{i\in I}$ spans $S_{\kappa}^{\ast}$ is similar. 
\end{proof}

\section{Discussions}\label{sect:discuss}
First, we want to remark that higher Rankin-Cohen brackets can be reduced to the two base levels, by \cite[Proposition 3.6]{XUeRC2024}, allowing us to restate Theorem \ref{thm:mainthm} below.
\begin{theorem}
    Let $n\geq1 $ be an integer. Assume that $\kappa\geq 4n+18$.
    \begin{enumerate}
        \item If $n$ is odd, then  $\{[E_{k+n-1},E_{\ell+n-1}]_1\}_{k,\ell}$ for even $k,\ell\geq4$ and $k+\ell+2n=\kappa$ spans $S_{\kappa}$.
        \item  If $n$ is even, then $\{E_{k+n}E_{\ell+n}-E_{\kappa}\}_{k,\ell}$ for even $k,\ell\geq4$ and $k+\ell+2n=\kappa$ spans $S_{\kappa}$.
    \end{enumerate}
\end{theorem}
One immediate question to ask is whether the non-singularity of $\mathbf{A}$ is preserved when the size of $\mathbf{A}$ gets larger; that is, if $\mathbf{A}$ captures more coefficients in the Eichler-Shimura relations. Some preliminary calculations show that once the size of $\mathbf{A}$ becomes larger than $\lceil\frac{3m}{2}\rceil$, the matrix reduction of $\mathbf{A}$ becomes very difficult to manage. 

It is also natural to ask if  our results can be extended from $\Sl_2(\mathbb{Z})$ to $\Gamma_0(N)$. In fact, Fukuhara and Yang \cite{Fukuharaperiod2} found explicit bases for $S_{\kappa}(\Gamma_0(2))$ and its dual space using a different approach. Note that periods do not form a basis in general since $\dim S_{\kappa}(\Gamma_0(N))\gg \kappa$ as $N\gg \kappa$, in which case one expects them to be linearly independent; see \cite[Theorem 2]{twistedLvaluesFukuhara}. Thus, the case of level one appears to be rather mysterious. 
Future work involves finding an explicit basis for $S_{\kappa}^{\ast}$ consisting of even periods, complementing the basis of odd periods in \cite{Fukuhara07}. More generally, we want to determine which periods are linearly independent. We record the following two conjectures:

\begin{conjecture}[{\cite[Conjecture 4.2]{evenperiods}}]\label{conj:even}
    Let $k_1>k_2>\cdots>k_n>\ell_n>\ell_{n-1}>\cdots>\ell_1\geq4$ be even integers such that $k_i+\ell_i+2=\kappa$ for $1\leq i\leq n$. If $n\leq \dim S_{\kappa}$, then the set of even periods $\{r_{\ell_i}\}_{i=1}^n$ in $S_{\kappa}^\ast$ is linearly independent. Equivalently, if $n\leq \dim S_{\kappa}$, then the set $\{[E_{k_i},E_{\ell_i}]_1\}_{i=1}^n$ is linearly independent. 
\end{conjecture}
\begin{conjecture}[{\cite[Conjecture 1.5]{oddperiods}}]\label{conj:odd}
  Let $k_1>k_2>\cdots>k_n\geq\ell_n>\ell_{n-1}>\cdots>\ell_1\geq4$ be even integers such that $k_i+\ell_i=\kappa$ for $1 \leq i \leq n$. If $n\leq \dim S_{\kappa}$, then the set of odd periods $\{r_{\ell_i-1}\}_{i=1}^n$ in $S_{\kappa}^\ast$ is linearly independent. Equivalently, if $n\leq \dim S_{\kappa}$, then the set $\{E_{k_i}E_{\ell_i}-E_{\kappa}\}_{i=1}^n$ is linearly independent.
\end{conjecture}
Conjectures \ref{conj:even} and \ref{conj:odd} claim that, unless required by dimension consideration, any set of even or odd periods is linearly independent. In particular, if $n\geq\dim S_{\kappa}$, then the aforementioned sets of periods span $S_{\kappa}^{\ast}$.
Therefore, Theorem \ref{thm:periodsspanningset} is an immediate corollary of Conjectures \ref{conj:even} and \ref{conj:odd}.
\section*{Acknowledgments}
This research was supported by NSF grant DMS-2349174. Hui Xue is supported by Simons Foundation grant MPS-TSM-00007911. We would like to thank Erick Ross for informing us of Lemma \ref{lem:binom-polynomial-vanish}.

\bibliographystyle{plain}

\providecommand{\bysame}{\leavevmode\hbox
to3em{\hrulefill}\thinspace}

\bibliography{biblref.bib}

\end{document}